\documentclass[a4paper,12pt]{article}
\usepackage[top=2.5cm,bottom=2.5cm,left=2.5cm,right=2.5cm]{geometry}
\usepackage{cite, amsmath, amssymb}
\usepackage[margin=1cm,%
font=small,%
format=hang,%
labelsep=period,%
labelfont=bf]{caption,subcaption}
\usepackage{amsmath,amsthm,amssymb,latexsym,mathrsfs,wrapfig,graphicx,float}
\usepackage{xcolor}

\usepackage{verbatim}

\newcommand{\noi}{\noindent}
\usepackage{tikz}
\usetikzlibrary{intersections, arrows, arrows.meta, automata, er, calc, mindmap, folding, patterns, decorations.markings, fit, snakes, shapes, matrix, positioning, shapes.geometric, through}
\definecolor{myblue}{rgb}{0.3, 0.4, 1}
\title{Tikz}

\newcommand{\vertex}{\node[vertex]}
\newtheorem{theorem}{Theorem}[section]
\newtheorem{definition}[theorem]{Definition}
\newtheorem{example}[theorem]{Example}

	\title{Line Completion Number of Grid Graph $P_n \times P_m$}
	\author{Joseph Varghese Kureethara,$^{1}$\footnote{Corresponding Author} ~Merin Sebastian$^2$\\
	\small Christ University, Bangalore, India\\
	\small $^1$frjoseph@christuniversity.in, $^{2}$merin.sebastian@maths.christuniversity.in
	}

	\date{}
    
\begin{document}
    \tikzstyle{arrow}=[thick, <-->, >=stealth]
    \tikzstyle{vertex}=[circle, draw, inner sep=1pt, minimum size=6pt]    	
    \maketitle
    
    \abstract{The concept of super line graph was introduced in the year 1995 by Bagga, Beineke and Varma. Given a graph  with at least $r$ edges, the super line graph of index $r$, $L_r(G)$, has as its vertices the sets of $r$-edges of $G$, with two adjacent if there is an edge in one set adjacent to an edge in the other set. The line completion number $lc(G)$  of a graph $G$ is the least positive integer $r$  for which $L_r(G)$ is a complete graph. In this paper, we find the line completion number of grid graph $P_n \times P_m$ for various cases of $n$ and $m$.}
    
    \baselineskip=0.30in
    
   \section{Introduction}
   \label{intro}
   Line graphs have been studied for over seventy years by now. The line graph transformation is one of the most widely studied of all graph transformations.
   In its long history, the concept has been rediscovered several times, with different names such as derived graphs, interchange graphs, edge-to-vertex duals etc. Line graphs can also be considered as intersection graphs. Several variations and generalizations of line graphs have been proposed and studied. These include the concepts of total graphs, path graphs, and others \cite{Kuree}.

\begin{figure}[h!]
	\centering
	\begin{tikzpicture}
	\vertex (u1) at (0,0) []{4};
	\vertex (u2) at (2,0) []{3};
	\vertex (u3) at (2,2) []{2};
	\vertex (u4) at (0,2) []{1};
	\vertex (ua) at (3,1) []{a};
	\vertex (ub) at (4,2) []{b};
	\vertex (uc) at (5,1) []{c};
	\vertex (ud) at (4,0) []{d};
	\vertex (ue) at (4,1) []{e};
	\path
	(u1) edge node [below] {$d$} (u2)
	(u2) edge node [right] {$c$} (u3)
	(u3) edge node [above] {$b$} (u4)
	(u4) edge node [left] {$a$} (u1)
	(u1) edge node [above] {$e$} (u3)
	(ua) edge (ub)
	(ua) edge (ud)
	(ua) edge (ue)
	(ud) edge (ue)
	(ud) edge (uc)
	(uc) edge (ue)
	(ub) edge (ue)
	(ub) edge (uc)
	;
	\end{tikzpicture}
	\caption{Diamond and its line graph}\label{fig:linegraph}
\end{figure}
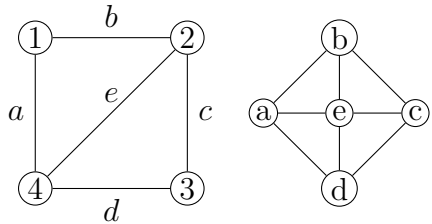
   \begin{definition}
	The Line Graph $L(G)$ of a graph $G$ is defined to have as its vertices the edges of $G$, with two being adjacent if the corresponding edges share a vertex in $G$.
\end{definition}
   In most generalizations of line graphs, the vertices of the new graph are taken to be another family of subgraphs of $G$, and adjacency is defined in terms of an appropriate intersection. For example, in the $r$-path graph $Pr (G)$ of a graph $G$, the vertices are the paths in $G$ of length $r$, with adjacency being defined as overlapping in a path of length $r-1$. We now see an interesting generalization of line graphs.
   
   \section{Super Line Graph }
   Given a graph $G$ with at least $r$ edges, \textbf{the super line graph (of index r) $L_r(G)$} has as its vertices the sets of $r$ edges of $G$, with two adjacent if there is an edge in one set adjacent to an edge in the other set.

   Figure \ref{fig:path} is the super line graph of (of index 2) $P_5$.
   \vskip 0.1cm
   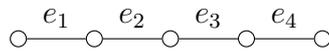
\begin{figure}[h!]
   	\centering
   	\begin{tikzpicture}
   	\vertex (u1) at (0,0) []{};
   	\vertex (u2) at (1,0) []{};
   	\vertex (u3) at (2,0) []{};
   	\vertex (u4) at (3,0) []{};
   	\vertex (u5) at (4,0) []{};
   	\path
   	(u1) edge node [above] {$e_1$} (u2)
   	(u2) edge node [above] {$e_2$} (u3)
   	(u3) edge node [above] {$e_3$} (u4)
   	(u4) edge node [above] {$e_4$} (u5)
   	;
   	\end{tikzpicture}
   	\caption{\label{fig:Path-$P_5$} Path $P_5$}
   \end{figure}

   \begin{figure}[h!]
   	\centering
   	\begin{tikzpicture}
   	\vertex (u1) at (0,0) []{\{$e_1, e_4$\}};
   	\vertex (u2) at (6,0) []{\{$e_1, e_3$\}};
   	\vertex (u3) at (0,3) []{\{$e_2, e_3$\}};
   	\vertex (u4) at (3,4) []{\{$e_3, e_4$\}};
   	\vertex (u5) at (3,-1) []{\{$e_1, e_2$\}};
   	\vertex (u6) at (6,3) []{\{$e_2, e_4$\}};
   	\path
   	(u1) edge (u2)
   	(u1) edge (u3)
   	(u1) edge (u4)
   	(u1) edge (u5)
   	(u1) edge (u6)
   	(u2) edge (u3)
   	(u2) edge (u4)
   	(u2) edge (u5)
   	(u2) edge (u6)
   	(u3) edge (u4)
   	(u3) edge (u5)
   	(u3) edge (u6)
   	(u4) edge (u5)
   	(u4) edge (u6)
   	(u5) edge (u6)
   	;
   	\end{tikzpicture}
   	\caption{\label{fig:path} Super Line Graph of index 2, $L_2(P_5)$}
   \end{figure}
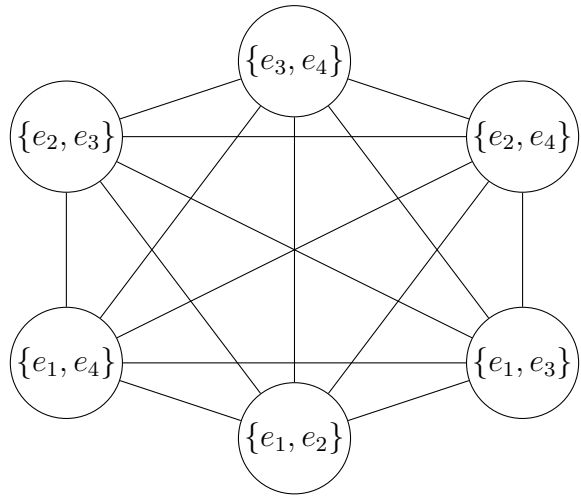

   \section{Line Completion Number of a Graph}
   
   Bagga et al. introduced in \cite{ks} the concept of super line graph in 1995. Bagga, Beineke and Varma contributed mostly in the study of super line graphs. See Table \ref{Tab:litslg} for a summary the literature.

   	\begin{table}[h]
   		\begin{center}
   			\begin{tabular}{|c|c|c|c|}
   				\hline 
   				\textbf{Authors}& \textbf{Concepts} & \textbf{Ref.}  & \textbf{Year} \\
   				\hline
   				\hline
   				Bagga et al.& super line graphs  &\cite{ks}  &1995 \\
   				\hline
   				Bagga et al.& properties of super line graphs  &\cite{bn}  &1995 \\
   				\hline
   				Bagga et al.& line completion number  &\cite{bagga}  &1995 \\
   				\hline
   				Bagga et al.& $L_2(G)$ &\cite{js}  &1999 \\   	   		
   				\hline
   				Bagga et al.&independence number, pancyclicity& \cite{lw} &1999 \\   	   		
   				\hline
   				Guti\'{e}rrez and Llad\'{o}&edge-residual number, dense graphs& \cite{as} &2002 \\
   				\hline
   				Bagga et al.&variations and generalizations of line graphs& \cite{JB} &2004 \\
   				\hline
   				Bagga and Ferrero&adjacency matrix, eigenvalues, isolated vertices & \cite{Ba} & 2005\\
   				\hline
   				Li et al.&path-comprehensiveness, vertex-pancyclicity& \cite{X} &2008 \\
   				\hline
   				Bagga et al.& lc of $K_n, P_n, C_n, F_n, W_n, M_n, K_{m,n}$ & \cite{Jay} & 2016\\
   				\hline
   				Tapadia and Waphare &lc of hypercubes& \cite{SA} &2019\\
   				\hline
   			\end{tabular}
   			\caption{Summary of the literature on Super Line Graph}
   			\label{Tab:litslg}
   		\end{center}
   	\end{table}
  
  Bagga et al. \cite{bagga} defined the line completion number of a graph as the least index $r$ for which super line graph becomes a complete graph.

 One of the important properties of super line graphs is that, if $L_r(G)$ is complete, so is $L_{r+1}(G)$. Bagga et al. determined the line completion numbers for various classes of graphs, viz., complete graph, path, cycle, fan, windmill, wheel, etc. Certain graphs are also characterized with the help of line completion number. In this paper we determine the line completion number of the Grid Graph.

 \section{Line Completion Number of Grid Graph}
Let $P_n$ be a path with n vertices and $n-1$ edges and $P_m$ be a path with $m$ vertices and $m-1$ edges.
 The grid graph $P_n \times P_m$ is the Cartesian product of $P_n$ and $P_m$ (See Figure \ref{fig:gridcommon}). The grid graph has $nm$ vertices and  $n(m-1)+m(n-1)=2mn-m-n$ edges. In $P_n \times P_m$, there are  $n$ $P_m$ paths and $m$ $P_n$ paths. Let the vertical paths be $P_m$'s and horizontal paths be $P_n$'s.
 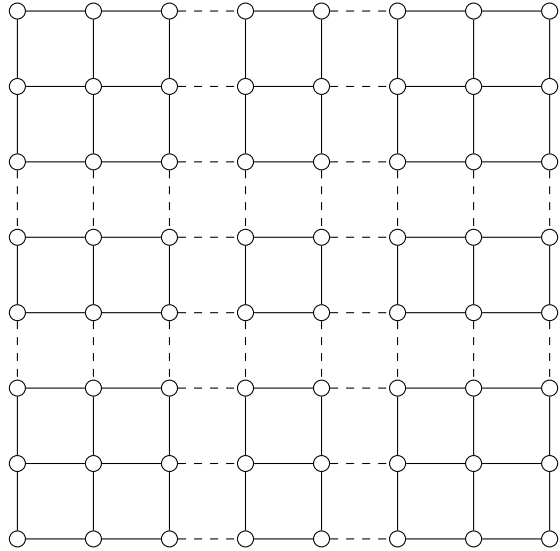
\begin{figure}
	\centering

 	\begin{tikzpicture}
 	\vertex (u1) at (0,4) []{};
 	\vertex (u2) at (0,3) []{};
 	\vertex (u3) at (0,2) []{};
 	\vertex (u4) at (0,1) []{};
 	\vertex (u5) at (0,0) []{};
 	\vertex (u6) at (0,-1) []{};
 	\vertex (u7) at (0,-2) []{};
 	\vertex (u8) at (0,-3) []{};
 	\vertex (u9) at (1,4) []{};
 	\vertex (u10) at (1,3) []{};
 	\vertex (u11) at (1,2) []{};
 	\vertex (u12) at (1,1) []{};
 	\vertex (u13) at (1,0) []{};
 	\vertex (u14) at (1,-1) []{};
 	\vertex (u15) at (1,-2) []{};
 	\vertex (u16) at (1,-3) []{};
 	\vertex (u17) at (2,4) []{};
 	\vertex (u18) at (2,3) []{};
 	\vertex (u19) at (2,2) []{};
 	\vertex (u20) at (2,1) []{};
 	\vertex (u21) at (2,0) []{};
 	\vertex (u22) at (2,-1) []{};
 	\vertex (u23) at (2,-2) []{};
 	\vertex (u24) at (2,-3) []{};
 	\vertex (u25) at (3,4) []{};
 	\vertex (u26) at (3,3) []{};
 	\vertex (u27) at (3,2) []{};
 	\vertex (u28) at (3,1) []{};
 	\vertex (u29) at (3,0) []{};
 	\vertex (u30) at (3,-1) []{};
 	\vertex (u31) at (3,-2) []{};
 	\vertex (u32) at (3,-3) []{};
 	\vertex (u33) at (4,4) []{};
 	\vertex (u34) at (4,3) []{};
 	\vertex (u35) at (4,2) []{};
 	\vertex (u36) at (4,1) []{};
 	\vertex (u37) at (4,0) []{};
 	\vertex (u38) at (4,-1) []{};
 	\vertex (u39) at (4,-2) []{};
 	\vertex (u40) at (4,-3) []{};
 	\vertex (u41) at (5,4) []{};
 	\vertex (u42) at (5,3) []{};
 	\vertex (u43) at (5,2) []{};
 	\vertex (u44) at (5,1) []{};
 	\vertex (u45) at (5,0) []{};
 	\vertex (u46) at (5,-1) []{};
 	\vertex (u47) at (5,-2) []{};
 	\vertex (u48) at (5,-3) []{};
 	\vertex (u49) at (6,4) []{};
 	\vertex (u50) at (6,3) []{};
 	\vertex (u51) at (6,2) []{};
 	\vertex (u52) at (6,1) []{};
 	\vertex (u53) at (6,0) []{};
 	\vertex (u54) at (6,-1) []{};
 	\vertex (u55) at (6,-2) []{};
 	\vertex (u56) at (6,-3) []{};
 	\vertex (u57) at (7,4) []{};
 	\vertex (u58) at (7,3) []{};
 	\vertex (u59) at (7,2) []{};
 	\vertex (u60) at (7,1) []{};
 	\vertex (u61) at (7,0) []{};
 	\vertex (u62) at (7,-1) []{};
 	\vertex (u63) at (7,-2) []{};
 	\vertex (u64) at (7,-3) []{};
 	\path
 	(u1) edge node [above] {} (u2)
 	(u2) edge node [above] {} (u3)
 	(u3) edge [dashed] (u4)
 	(u4) edge node [above] {} (u5)
 	(u5) edge [dashed] (u6)
 	(u6) edge (u7)
 	(u7) edge (u8)
 	(u9) edge node [above] {} (u10)
 	(u10) edge node [above] {} (u11)
 	(u11) edge [dashed] (u12)
 	(u12) edge node [above] {} (u13)
 	(u13) edge [dashed] (u14)
 	(u14) edge (u15)
 	(u15) edge (u16)
 	(u17) edge node [above] {} (u18)
 	(u18) edge node [above] {} (u19)
 	(u19) edge [dashed] (u20)
 	(u20) edge node [above] {} (u21)
 	(u21) edge [dashed] (u22)
 	(u22) edge (u23)
 	(u23) edge (u24)
 	(u25) edge node [above] {} (u26)
 	(u26) edge node [above] {} (u27)
 	(u27) edge [dashed] (u28)
 	(u28) edge node [above] {} (u29)
 	(u29) edge [dashed] (u30)
 	(u30) edge (u31)
 	(u31) edge (u32)
 	(u33) edge node [above] {} (u34)
 	(u34) edge node [above] {} (u35)
 	(u35) edge [dashed] (u36)
 	(u36) edge node [above] {} (u37)
 	(u37) edge [dashed] (u38)
 	(u38) edge (u39)
 	(u39) edge (u40)
 	(u41) edge node [above] {} (u42)
 	(u42) edge node [above] {} (u43)
 	(u43) edge [dashed] (u44)
 	(u44) edge node [above] {} (u45)
 	(u45) edge [dashed] (u46)
 	(u46) edge (u47)
 	(u47) edge (u48)
 	(u49) edge node [above] {} (u50)
 	(u50) edge node [above] {} (u51)
 	(u51) edge [dashed] (u52)
 	(u52) edge node [above] {} (u53)
 	(u53) edge [dashed] (u54)
 	(u54) edge (u55)
 	(u55) edge (u56)
 	(u57) edge node [above] {} (u58)
 	(u58) edge node [above] {} (u59)
 	(u59) edge [dashed] (u60)
 	(u60) edge node [above] {} (u61)
 	(u61) edge [dashed] (u62)
 	(u62) edge (u63)
 	(u63) edge (u64)
 	(u1) edge node [above] {} (u9)
 	(u9) edge node [above] {} (u17)
 	(u17) edge [dashed] (u25)
 	(u25) edge node [above] {} (u33)
 	(u33) edge [dashed] (u41)
 	(u41) edge (u49)
 	(u49) edge (u57)
 	(u2) edge node [above] {} (u10)
 	(u10) edge node [above] {} (u18)
 	(u18) edge [dashed] (u26)
 	(u26) edge node [above] {} (u34)
 	(u34) edge [dashed] (u42)
 	(u42) edge (u50)
 	(u50) edge (u58)
 	(u3) edge node [above] {} (u11)
 	(u11) edge node [above] {} (u19)
 	(u19) edge [dashed] (u27)
 	(u27) edge node [above] {} (u35)
 	(u35) edge [dashed] (u43)
 	(u43) edge (u51)
 	(u51) edge (u59)
 	(u4) edge node [above] {} (u12)
 	(u12) edge node [above] {} (u20)
 	(u20) edge [dashed] (u28)
 	(u28) edge node [above] {} (u36)
 	(u36) edge [dashed] (u44)
 	(u44) edge (u52)
 	(u52) edge (u60)
 	(u5) edge node [above] {} (u13)
 	(u13) edge node [above] {} (u21)
 	(u21) edge [dashed] (u29)
 	(u29) edge node [above] {} (u37)
 	(u37) edge [dashed] (u45)
 	(u45) edge (u53)
 	(u53) edge (u61)
 	(u6) edge node [above] {} (u14)
 	(u14) edge node [above] {} (u22)
 	(u22) edge [dashed] (u30)
 	(u30) edge node [above] {} (u38)
 	(u38) edge [dashed] (u46)
 	(u46) edge (u54)
 	(u54) edge (u62)
 	(u7) edge node [above] {} (u15)
 	(u15) edge node [above] {} (u23)
 	(u23) edge [dashed] (u31)
 	(u31) edge node [above] {} (u39)
 	(u39) edge [dashed] (u47)
 	(u47) edge (u55)
 	(u55) edge (u63)
 	(u8) edge node [above] {} (u16)
 	(u16) edge node [above] {} (u24)
 	(u24) edge [dashed] (u32)
 	(u32) edge node [above] {} (u40)
 	(u40) edge [dashed] (u48)
 	(u48) edge (u56)
 	(u56) edge (u64)
 	
 	;
 	\end{tikzpicture}

	\caption{A general finite grid}
	\label{fig:gridcommon}
\end{figure}

 \begin{theorem}
 Let $P_n$ be a path with n vertices and $P_m$ be a path with m vertices. The line completion number of grid graph $P_n \times P_m$, 
 	$$
 	lc(P_n \times P_m)=
 	\begin{cases}
 	0 &\text{when }n = m =1\\
 	\lfloor \frac{\max{(n, m)}}{2} \rfloor & \text{when exactly one of n or m is 1}\\
 	mn+1-\frac{(m+n)}{2}-\frac{\min{(m, n)}}{2} &\text{when n and m are even}\\
 	mn-\frac{(m+n)}{2}-\frac{\min{(m, n)}+1}{2} &\text{when  }n\ne 1\text{  and }m\ne 1 \text{ are odd}\\
 	mn+1-\min{(m, n)}-\lceil\frac{\max{(m, n)}}{2}\rceil &\text{when  }n\ne 1\text{  and }m\ne 1 \text{ are of opposite parity}
 	\end{cases}
 	$$
 \end{theorem}

 \begin{proof}
 	We prove the theorem by analysing all the five cases.
 	Let $P_n$ be the path with $n$ vertices and $n-1$ edges and $P_m$ be the path with $m$ vertices and $m-1$ edges. The grid graphs $P_n \times P_m$ and $P_m \times P_n$ are isomorphic graphs.
 	
 	To find the line completion number of $P_n \times P_m$, our aim is to partition the vertex set of $P_n \times P_m$ into two subsets $A$ and $B$, so that the subgraphs induced by the two sets $A$ and $B$, viz., $A_k$ and $B_k$, respectively, are not adjacent in $P_n \times P_m$. While doing this, we ensure to simultaneously maximize the number of edges in  $A_k$ and $B_k$ and minimize the edges that are not in either of these subgraphs. Once we obtain these optimal graphs, adding an edge that is not in $A_k$ or $B_k$ would fetch us the required result. 
 	
 	For the grid graph  $P_n \times P_m$, in a Cartesian grid presentation, let the horizontal and the vertical paths be the $P_n$'s and the $P_m$'s, respectively. 
 	 	
 	\section*{Case 1: $n = m =1$}
 	
  	The graph $P_1 \times P_1$ is $P_1$ with no edges. Hence, the line completion number trivially is 0.
 	
 	\section*{Case 2: when exactly one of $n$ or $m$ is 1}
 	
 	The graph in this case is either $P_1 \times P_m$ or $P_n \times P_1$. It is nothing but the path graph. Hence, without loss of generality we  consider the graph $P_n$. 
 	
 	\subsection*{Subcase 2.1: $n$ is odd}

The graph $P_n$ has $n-1$ edges. Then if we consider a subset of the edge set of $P_n$ with the cardinality ranging from 1 to $\frac{(n-1)}{2}-1$, there is at least one other subset of the same size that is not adjacent to it. Hence in all such cases, the super line graph is not a complete graph. The two subsets of edges that are not assured to be non-adjacent to each other of the same size are the subsets of consecutive edges from the pendant vertices of the path. See the illustration in Figure \ref{fig:path1}.
 	\begin{figure}
 		\centering
 		\begin{tikzpicture}
 		\vertex [fill=black,circle,inner sep=2pt] (u1) at (0,0) []{};
 		\vertex [fill=black,circle,inner sep=2pt] (u2) at (1,0) []{};
 		\vertex [fill=black,circle,inner sep=2pt] (u3) at (2,0) []{};
 		\vertex [fill=black,circle,inner sep=2pt] (u4) at (3,0) []{};
 		\vertex (u5) at (4,0) []{};
 		\vertex [fill=black,circle,inner sep=2pt] (u6) at (5,0) []{};
 		\vertex [fill=black,circle,inner sep=2pt] (u7) at (6,0) []{};
 		\vertex [fill=black,circle,inner sep=2pt] (u8) at (7,0) []{};
 		\vertex [fill=black,circle,inner sep=2pt] (u9) at (8,0) []{};
 		\path
 		(u1) edge node [above] {} (u2)
 		(u2) edge node [above] {} (u3)
 		(u3) edge node [above] {} (u4)
 		(u4) edge node [above] {} (u5)
 		(u5) edge node [above] {} (u6)
 		(u6) edge node [above] {} (u7)
 		(u7) edge node [above] {} (u8)
 		(u8) edge node [above] {} (u9)
 		;
 		\end{tikzpicture}
 		
 		\caption{A path with odd number of vertices}
 		\label{fig:path1}
 	\end{figure}
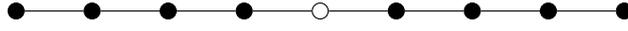
 	
 	However, if we take a subset of the edge set of size $\frac{(n-1)}{2}$, it is adjacent to every subset of size $\frac{(n-1)}{2}$.
 	
 	Thus, the line completion number of $P_n$ is $\frac{(n-1)}{2}-1$+1=$\frac{(n-1)}{2}$.
 	
 	\subsection*{Subcase 2.2: $n$ is even}
 	
 	Continuing in the same line of argument, we can see that the two connected subgraphs each with $\frac{n}{2}$ vertices formed with the two pendant vertices in either of the graphs are non-adjacent. Each of them is of size $\frac{(n-2)}{2}$. See the illustration in Figure \ref{fig:path2}.
 	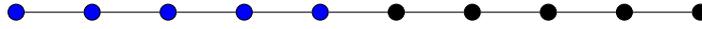
\begin{figure}
 		\centering
 		\begin{tikzpicture}
 		\vertex [fill=blue,circle,inner sep=2pt] (u1) at (0,0) []{};
 		\vertex [fill=blue,circle,inner sep=2pt] (u2) at (1,0) []{};
 		\vertex [fill=blue,circle,inner sep=2pt] (u3) at (2,0) []{};
 		\vertex [fill=blue,circle,inner sep=2pt] (u4) at (3,0) []{};
 		\vertex (u5) [fill=blue,circle,inner sep=2pt] at (4,0) []{};
 		\vertex [fill=black,circle,inner sep=2pt] (u6) at (5,0) []{};
 		\vertex [fill=black,circle,inner sep=2pt] (u7) at (6,0) []{};
 		\vertex [fill=black,circle,inner sep=2pt] (u8) at (7,0) []{};
 		\vertex [fill=black,circle,inner sep=2pt] (u9) at (8,0) []{};
 		\vertex [fill=black,circle,inner sep=2pt] (u10) at (9,0) []{};
 		\path
 		(u1) edge node [above] {} (u2)
 		(u2) edge node [above] {} (u3)
 		(u3) edge node [above] {} (u4)
 		(u4) edge node [above] {} (u5)
 		(u5) edge node [above] {} (u6)
 		(u6) edge node [above] {} (u7)
 		(u7) edge node [above] {} (u8)
 		(u8) edge node [above] {} (u9)
 		(u9) edge node [above] {} (u10)
 		;
 		\end{tikzpicture}
 		
 		\caption{A path with even number of vertices}
 		\label{fig:path2}
 	\end{figure}
 	
 	Thus, the line completion number of $P_n$ is $\frac{(n-2)}{2}$+1=$\frac{n}{2}$.
 	
 	Combining both the cases, we conclude that, the line completion number of a path $P_n \times P_m$ when exactly one of $n$ or $m$ is 1, is $\lfloor\frac{\max(n,m)}{2}\rfloor$. This result is equivalent to the result found in \cite{Jay}.

 	\subsection*{Slicing of the Grid: Various Scenarios}
 	
 	As indicated in the proof strategy, we need to slice the grid $P_n \times P_m$. In doing so we consider the following cases.
 	
 	\subsubsection*{Scenario 1: $n$ is even}

 	Here the vertical slicing of the graph through the central edges of $P_n$'s will give us two $P_{\frac{n}{2}}\times P_m$'s that are not adjacent to each other and an $mK_2$. This is illustrated in the Figure \ref{fig:gridv01}. 
 	
 		\begin{figure}
 		
 		\centering
 		\begin{tikzpicture}
 		\vertex (u1) at (0,4) []{};
 		\vertex (u2) at (0,3) []{};
 		\vertex (u3) at (0,2) []{};
 		\vertex (u4) at (0,1) []{};
 		\vertex (u5) at (0,0) []{};
 		
 		\vertex (u9) at (1,4) []{};
 		\vertex (u10) at (1,3) []{};
 		\vertex (u11) at (1,2) []{};
 		\vertex (u12) at (1,1) []{};
 		\vertex (u13) at (1,0) []{};
 		
 		\vertex (u17) at (2,4) []{};
 		\vertex (u18) at (2,3) []{};
 		\vertex (u19) at (2,2) []{};
 		\vertex (u20) at (2,1) []{};
 		\vertex (u21) at (2,0) []{};
 		
 		\vertex (u25) at (3,4) []{};
 		\vertex (u26) at (3,3) []{};
 		\vertex (u27) at (3,2) []{};
 		\vertex (u28) at (3,1) []{};
 		\vertex (u29) at (3,0) []{};
 		
 		\vertex (u33) at (4,4) []{};
 		\vertex (u34) at (4,3) []{};
 		\vertex (u35) at (4,2) []{};
 		\vertex (u36) at (4,1) []{};
 		\vertex (u37) at (4,0) []{};
 		
 		\vertex (u41) at (5,4) []{};
 		\vertex (u42) at (5,3) []{};
 		\vertex (u43) at (5,2) []{};
 		\vertex (u44) at (5,1) []{};
 		\vertex (u45) at (5,0) []{};
 		
 		\vertex (u49) at (6,4) []{};
 		\vertex (u50) at (6,3) []{};
 		\vertex (u51) at (6,2) []{};
 		\vertex (u52) at (6,1) []{};
 		\vertex (u53) at (6,0) []{};
 		
 		\vertex (u57) at (7,4) []{};
 		\vertex (u58) at (7,3) []{};
 		\vertex (u59) at (7,2) []{};
 		\vertex (u60) at (7,1) []{};
 		\vertex (u61) at (7,0) []{};

 		\path
 		(u1) edge node [above] {} (u2)
 		(u2) edge node [above] {} (u3)
 		(u3) edge node [above] {} (u4)
 		(u4) edge node [above] {} (u5)
 		
 		(u9) edge node [above] {} (u10)
 		(u10) edge node [above] {} (u11)
 		(u11) edge node [above] {}(u12)
 		(u12) edge node [above] {} (u13)
 		
 		(u17) edge node [above] {} (u18)
 		(u18) edge node [above] {} (u19)
 		(u19) edge node [above] {} (u20)
 		(u20) edge node [above] {} (u21)
 		
 		(u25) edge node [above] {} (u26)
 		(u26) edge node [above] {} (u27)
 		(u27) edge node [above] {} (u28)
 		(u28) edge node [above] {} (u29)
 		
 		(u33) edge node [above] {} (u34)
 		(u34) edge node [above] {} (u35)
 		(u35) edge node [above] {} (u36)
 		(u36) edge node [above] {} (u37)
 		
 		(u41) edge node [above] {} (u42)
 		(u42) edge node [above] {} (u43)
 		(u43) edge node [above] {} (u44)
 		(u44) edge node [above] {} (u45)
 		
 		(u49) edge node [above] {} (u50)
 		(u50) edge node [above] {} (u51)
 		(u51) edge node [above] {} (u52)
 		(u52) edge node [above] {} (u53)
 		
 		(u57) edge node [above] {} (u58)
 		(u58) edge node [above] {} (u59)
 		(u59) edge node [above] {} (u60)
 		(u60) edge node [above] {} (u61)
 		
 		(u1) edge node [above] {} (u9)
 		(u9) edge node [above] {} (u17)
 		(u17) edge node [above] {} (u25)
 		(u25) edge [line width=3pt] (u33)
 		(u33) edge node [above] {} (u41)
 		(u41) edge (u49)
 		(u49) edge (u57)
 		(u2) edge node [above] {} (u10)
 		(u10) edge node [above] {} (u18)
 		(u18) edge node [above] {} (u26)
 		(u26) edge [line width=3pt] (u34)
 		(u34) edge node [above] {} (u42)
 		(u42) edge (u50)
 		(u50) edge (u58)
 		(u3) edge node [above] {} (u11)
 		(u11) edge node [above] {} (u19)
 		(u19) edge node [above] {} (u27)
 		(u27) edge [line width=3pt] (u35)
 		(u35) edge node [above] {} (u43)
 		(u43) edge (u51)
 		(u51) edge (u59)
 		(u4) edge node [above] {} (u12)
 		(u12) edge node [above] {} (u20)
 		(u20) edge node [above] {} (u28)
 		(u28) edge [line width=3pt] (u36)
 		(u36) edge node [above] {} (u44)
 		(u44) edge (u52)
 		(u52) edge (u60)
 		(u5) edge node [above] {} (u13)
 		(u13) edge node [above] {} (u21)
 		(u21) edge node [above] {} (u29)
 		(u29) edge [line width=3pt] (u37)
 		(u37) edge node [above] {} (u45)
 		(u45) edge (u53)
 		(u53) edge (u61)
 		;
 		\end{tikzpicture}
 		
 		\caption{A vertical slicing}
 		\label{fig:gridv01}
 	\end{figure}
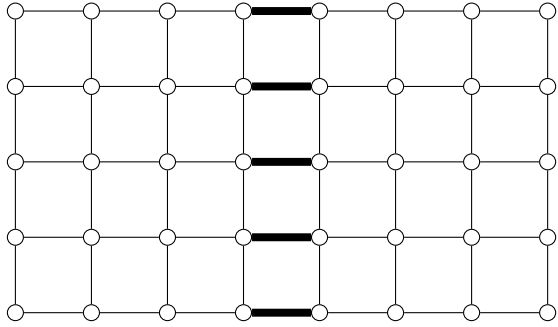
 	
 	\subsubsection*{Scenario 2: $m$ is even}
 	 
 	If $m$ is even a horizontal slicing of the graph can be done as illustrated in Figure \ref{fig:gridh01}. This would generate an $nK_2$ whose edges are not in the slices.
 		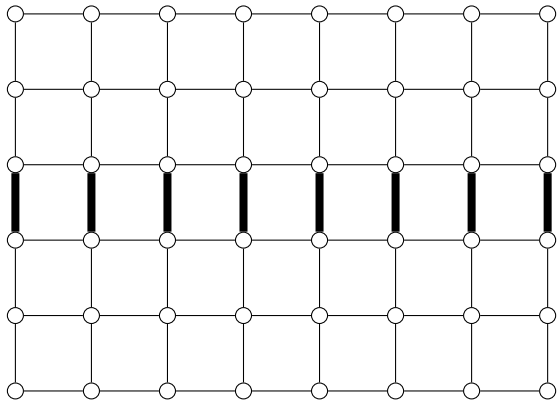
\begin{figure}
 		\centering
 			\begin{tikzpicture}
 	\vertex (u1) at (0,4) []{};
 	\vertex (u2) at (0,3) []{};
 	\vertex (u3) at (0,2) []{};
 	\vertex (u4) at (0,1) []{};
 	\vertex (u5) at (0,0) []{};
 	\vertex (u6) at (0,-1) []{};
 	
 	\vertex (u9) at (1,4) []{};
 	\vertex (u10) at (1,3) []{};
 	\vertex (u11) at (1,2) []{};
 	\vertex (u12) at (1,1) []{};
 	\vertex (u13) at (1,0) []{};
 	\vertex (u14) at (1,-1) []{};
 
 	\vertex (u17) at (2,4) []{};
 	\vertex (u18) at (2,3) []{};
 	\vertex (u19) at (2,2) []{};
 	\vertex (u20) at (2,1) []{};
 	\vertex (u21) at (2,0) []{};
 	\vertex (u22) at (2,-1) []{};
 	
 	\vertex (u25) at (3,4) []{};
 	\vertex (u26) at (3,3) []{};
 	\vertex (u27) at (3,2) []{};
 	\vertex (u28) at (3,1) []{};
 	\vertex (u29) at (3,0) []{};
 	\vertex (u30) at (3,-1) []{};

 	\vertex (u33) at (4,4) []{};
 	\vertex (u34) at (4,3) []{};
 	\vertex (u35) at (4,2) []{};
 	\vertex (u36) at (4,1) []{};
 	\vertex (u37) at (4,0) []{};
 	\vertex (u38) at (4,-1) []{};
 	
 	\vertex (u41) at (5,4) []{};
 	\vertex (u42) at (5,3) []{};
 	\vertex (u43) at (5,2) []{};
 	\vertex (u44) at (5,1) []{};
 	\vertex (u45) at (5,0) []{};
 	\vertex (u46) at (5,-1) []{};
 	
 	\vertex (u49) at (6,4) []{};
 	\vertex (u50) at (6,3) []{};
 	\vertex (u51) at (6,2) []{};
 	\vertex (u52) at (6,1) []{};
 	\vertex (u53) at (6,0) []{};
 	\vertex (u54) at (6,-1) []{};
 	
 	\vertex (u57) at (7,4) []{};
 	\vertex (u58) at (7,3) []{};
 	\vertex (u59) at (7,2) []{};
 	\vertex (u60) at (7,1) []{};
 	\vertex (u61) at (7,0) []{};
 	\vertex (u62) at (7,-1) []{};
 	\path
 	(u1) edge node [above] {} (u2)
 	(u2) edge node [above] {} (u3)
 	(u3) edge[line width= 3pt] (u4)
 	(u4) edge node [above] {} (u5)
 	(u5) edge node [above] {} (u6)

 	(u9) edge node [above] {} (u10)
 	(u10) edge node [above] {} (u11)
 	(u11) edge[line width= 3pt] (u12)
 	(u12) edge node [above] {} (u13)
 	(u13) edge node [above] {} (u14)
 	
 	(u17) edge node [above] {} (u18)
 	(u18) edge node [above] {} (u19)
 	(u19) edge[line width= 3pt] (u20)
 	(u20) edge node [above] {} (u21)
 	(u21) edge node [above] {} (u22)
 	
 	(u25) edge node [above] {} (u26)
 	(u26) edge node [above] {} (u27)
 	(u27) edge[line width= 3pt] (u28)
 	(u28) edge node [above] {} (u29)
 	(u29) edge node [above] {} (u30)
 	
 	(u33) edge node [above] {} (u34)
 	(u34) edge node [above] {} (u35)
 	(u35) edge[line width= 3pt] (u36)
 	(u36) edge node [above] {} (u37)
 	(u37) edge node [above] {} (u38)
 	
 	(u41) edge node [above] {} (u42)
 	(u42) edge node [above] {} (u43)
 	(u43) edge[line width= 3pt] (u44)
 	(u44) edge node [above] {} (u45)
 	(u45) edge node [above] {} (u46)
 	
 	(u49) edge node [above] {} (u50)
 	(u50) edge node [above] {} (u51)
 	(u51) edge[line width= 3pt] (u52)
 	(u52) edge node [above] {} (u53)
 	(u53) edge node [above] {} (u54)
 	
 	(u57) edge node [above] {} (u58)
 	(u58) edge node [above] {} (u59)
 	(u59) edge[line width= 3pt] (u60)
 	(u60) edge node [above] {} (u61)
 	(u61) edge node [above] {} (u62)
 	
 	(u1) edge node [above] {} (u9)
 	(u9) edge node [above] {} (u17)
 	(u17) edge node [above] {} (u25)
 	(u25) edge node [above] {} (u33)
 	(u33) edge node [above] {} (u41)
 	(u41) edge (u49)
 	(u49) edge (u57)
 	(u2) edge node [above] {} (u10)
 	(u10) edge node [above] {} (u18)
 	(u18) edge node [above] {} (u26)
 	(u26) edge node [above] {} (u34)
 	(u34) edge node [above] {} (u42)
 	(u42) edge (u50)
 	(u50) edge (u58)
 	(u3) edge node [above] {} (u11)
 	(u11) edge node [above] {} (u19)
 	(u19) edge node [above] {} (u27)
 	(u27) edge node [above] {} (u35)
 	(u35) edge node [above] {} (u43)
 	(u43) edge (u51)
 	(u51) edge (u59)
 	(u4) edge node [above] {} (u12)
 	(u12) edge node [above] {} (u20)
 	(u20) edge node [above] {} (u28)
 	(u28) edge node [above] {} (u36)
 	(u36) edge node [above] {} (u44)
 	(u44) edge (u52)
 	(u52) edge (u60)
 	(u5) edge node [above] {} (u13)
 	(u13) edge node [above] {} (u21)
 	(u21) edge node [above] {} (u29)
 	(u29) edge node [above] {} (u37)
 	(u37) edge node [above] {} (u45)
 	(u45) edge (u53)
 	(u53) edge (u61)
 	(u6) edge node [above] {} (u14)
 	(u14) edge node [above] {} (u22)
 	(u22) edge node [above] {} (u30)
 	(u30) edge node [above] {} (u38)
 	(u38) edge node [above] {} (u46)
 	(u46) edge (u54)
 	(u54) edge (u62)
 	
 	;
 	\end{tikzpicture}
 
 		\caption{A horizontal slicing}
 		\label{fig:gridh01}
 	\end{figure}

 	\subsubsection*{Scenario 3: $n$ is odd}
 		
 	Since $n$ is odd, there are two central edges for the horizontal $P_n$'s. Hence, to get two slices of the grid that are equal in size, we do a slicing as presented in Figure \ref{fig:gridv02}.
 	
 	\begin{figure}
 		\centering
 		\begin{tikzpicture}
 	\vertex (u1) at (0,4) []{};
 	\vertex (u2) at (0,3) []{};
 	\vertex (u3) at (0,2) []{};
 	\vertex (u4) at (0,1) []{};
 	\vertex (u5) at (0,0) []{};
 	\vertex (u6) at (0,-1) []{};
 	
 	\vertex (u9) at (1,4) []{};
 	\vertex (u10) at (1,3) []{};
 	\vertex (u11) at (1,2) []{};
 	\vertex (u12) at (1,1) []{};
 	\vertex (u13) at (1,0) []{};
 	\vertex (u14) at (1,-1) []{};
 
 	\vertex (u17) at (2,4) []{};
 	\vertex (u18) at (2,3) []{};
 	\vertex (u19) at (2,2) []{};
 	\vertex (u20) at (2,1) []{};
 	\vertex (u21) at (2,0) []{};
 	\vertex (u22) at (2,-1) []{};
 	
 	\vertex (u25) at (3,4) []{};
 	\vertex (u26) at (3,3) []{};
 	\vertex (u27) at (3,2) []{};
 	\vertex (u28) at (3,1) []{};
 	\vertex (u29) at (3,0) []{};
 	\vertex (u30) at (3,-1) []{};

 	\vertex (u33) at (4,4) []{};
 	\vertex (u34) at (4,3) []{};
 	\vertex (u35) at (4,2) []{};
 	\vertex (u36) at (4,1) []{};
 	\vertex (u37) at (4,0) []{};
 	\vertex (u38) at (4,-1) []{};
 	
 	\vertex (u41) at (5,4) []{};
 	\vertex (u42) at (5,3) []{};
 	\vertex (u43) at (5,2) []{};
 	\vertex (u44) at (5,1) []{};
 	\vertex (u45) at (5,0) []{};
 	\vertex (u46) at (5,-1) []{};
 	
 	\vertex (u49) at (6,4) []{};
 	\vertex (u50) at (6,3) []{};
 	\vertex (u51) at (6,2) []{};
 	\vertex (u52) at (6,1) []{};
 	\vertex (u53) at (6,0) []{};
 	\vertex (u54) at (6,-1) []{};

 	\path
 	(u1) edge node [above] {} (u2)
 	(u2) edge node [above] {} (u3)
 	(u3) edge node [above] {} (u4)
 	(u4) edge node [above] {} (u5)
 	(u5) edge node [above] {} (u6)

 	(u9) edge node [above] {} (u10)
 	(u10) edge node [above] {} (u11)
 	(u11) edge node [above] {} (u12)
 	(u12) edge node [above] {} (u13)
 	(u13) edge node [above] {} (u14)
 	
 	(u17) edge node [above] {} (u18)
 	(u18) edge node [above] {} (u19)
 	(u19) edge node [above] {} (u20)
 	(u20) edge node [above] {} (u21)
 	(u21) edge node [above] {} (u22)
 	
 	(u25) edge node [above] {} (u26)
 	(u26) edge node [above] {} (u27)
 	(u27) edge [line width= 3pt]  (u28)
 	(u28) edge node [above] {} (u29)
 	(u29) edge node [above] {} (u30)
 	
 	(u33) edge node [above] {} (u34)
 	(u34) edge node [above] {} (u35)
 	(u35) edge node [above] {} (u36)
 	(u36) edge node [above] {} (u37)
 	(u37) edge node [above] {} (u38)
 	
 	(u41) edge node [above] {} (u42)
 	(u42) edge node [above] {} (u43)
 	(u43) edge node [above] {} (u44)
 	(u44) edge node [above] {} (u45)
 	(u45) edge node [above] {} (u46)
 	
 	(u49) edge node [above] {} (u50)
 	(u50) edge node [above] {} (u51)
 	(u51) edge node [above] {} (u52)
 	(u52) edge node [above] {} (u53)
 	(u53) edge node [above] {} (u54)

 	(u1) edge node [above] {} (u9)
 	(u9) edge node [above] {} (u17)
 	(u17) edge node [above] {} (u25)
 	(u25) edge [line width= 3pt] (u33)
 	(u33) edge node [above] {} (u41)
 	(u41) edge (u49)
 	
 	(u2) edge node [above] {} (u10)
 	(u10) edge node [above] {} (u18)
 	(u18) edge node [above] {} (u26)
 	(u26) edge [line width= 3pt] (u34)
 	(u34) edge node [above] {} (u42)
 	(u42) edge (u50)
 	
 	(u3) edge node [above] {} (u11)
 	(u11) edge node [above] {} (u19)
 	(u19) edge node [above] {} (u27)
 	(u27) edge [line width= 3pt] (u35)
 	(u35) edge node [above] {} (u43)
 	(u43) edge (u51)
 	
 	(u4) edge node [above] {} (u12)
 	(u12) edge node [above] {} (u20)
 	(u20) edge [line width= 3pt]  (u28)
 	(u28) edge node [above] {} (u36)
 	(u36) edge node [above] {} (u44)
 	(u44) edge (u52)
 	
 	(u5) edge node [above] {} (u13)
 	(u13) edge node [above] {} (u21)
 	(u21) edge [line width= 3pt] (u29)
 	(u29) edge node [above] {} (u37)
 	(u37) edge node [above] {} (u45)
 	(u45) edge (u53)
 	
 	(u6) edge node [above] {} (u14)
 	(u14) edge node [above] {} (u22)
 	(u22) edge [line width= 3pt] (u30)
 	(u30) edge node [above] {} (u38)
 	(u38) edge node [above] {} (u46)
 	(u46) edge (u54)
  	
 	;
 	\end{tikzpicture}
 	
 		\caption{An almost verticial slicing}
 		\label{fig:gridv02}
 	\end{figure}
 	\subsubsection*{Scenario 4: $m$ is odd}
 	
 	Horizontal slicing of the grid graphs also can be done in an analogous manner. (See Figure \ref{fig:gridh02}.)

 	\subsubsection*{Scenario 5: $m$ and $n$ are odd}
 	
 	In this case, the number of edges that do not form part of the partition is $m+3$ and $n+3$, based on the vertical and horizontal slicing, respectively. See Figure \ref{fig:gridv03} and \ref{fig:gridh03}.
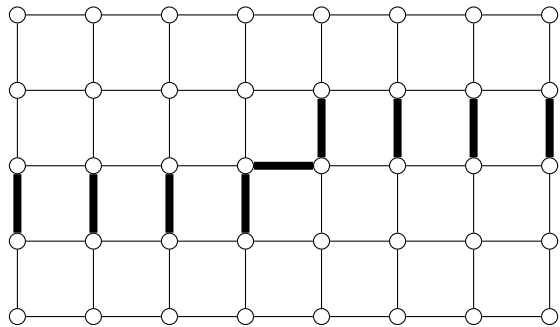
\begin{figure}
	\centering
	\begin{tikzpicture}
 	\vertex (u1) at (0,4) []{};
 	\vertex (u2) at (0,3) []{};
 	\vertex (u3) at (0,2) []{};
 	\vertex (u4) at (0,1) []{};
 	\vertex (u5) at (0,0) []{};

 	\vertex (u9) at (1,4) []{};
 	\vertex (u10) at (1,3) []{};
 	\vertex (u11) at (1,2) []{};
 	\vertex (u12) at (1,1) []{};
 	\vertex (u13) at (1,0) []{};

 	\vertex (u17) at (2,4) []{};
 	\vertex (u18) at (2,3) []{};
 	\vertex (u19) at (2,2) []{};
 	\vertex (u20) at (2,1) []{};
 	\vertex (u21) at (2,0) []{};
 	
 	\vertex (u25) at (3,4) []{};
 	\vertex (u26) at (3,3) []{};
 	\vertex (u27) at (3,2) []{};
 	\vertex (u28) at (3,1) []{};
 	\vertex (u29) at (3,0) []{};

 	\vertex (u33) at (4,4) []{};
 	\vertex (u34) at (4,3) []{};
 	\vertex (u35) at (4,2) []{};
 	\vertex (u36) at (4,1) []{};
 	\vertex (u37) at (4,0) []{};

 	\vertex (u41) at (5,4) []{};
 	\vertex (u42) at (5,3) []{};
 	\vertex (u43) at (5,2) []{};
 	\vertex (u44) at (5,1) []{};
 	\vertex (u45) at (5,0) []{};

 	\vertex (u49) at (6,4) []{};
 	\vertex (u50) at (6,3) []{};
 	\vertex (u51) at (6,2) []{};
 	\vertex (u52) at (6,1) []{};
 	\vertex (u53) at (6,0) []{};

 	\vertex (u57) at (7,4) []{};
 	\vertex (u58) at (7,3) []{};
 	\vertex (u59) at (7,2) []{};
 	\vertex (u60) at (7,1) []{};
 	\vertex (u61) at (7,0) []{};
 	
 	\path
 	(u1) edge node [above] {} (u2)
 	(u2) edge node [above] {} (u3)
 	(u3) edge[line width= 3pt] (u4)
 	(u4) edge node [above] {} (u5)

 	(u9) edge node [above] {} (u10)
 	(u10) edge node [above] {} (u11)
 	(u11) edge[line width= 3pt] (u12)
 	(u12) edge node [above] {} (u13)

 	(u17) edge node [above] {} (u18)
 	(u18) edge node [above] {} (u19)
 	(u19) edge[line width= 3pt] (u20)
 	(u20) edge node [above] {} (u21)

 	(u25) edge node [above] {} (u26)
 	(u26) edge node [above] {} (u27)
 	(u27) edge[line width= 3pt] (u28)
 	(u28) edge node [above] {} (u29)

 	(u33) edge node [above] {} (u34)
 	(u34) edge [line width= 3pt] (u35)
 	(u35) edge node [above] {} (u36)
 	(u36) edge node [above] {} (u37)

 	(u41) edge node [above] {} (u42)
 	(u42) edge [line width= 3pt] (u43)
 	(u43) edge node [above] {} (u44)
 	(u44) edge node [above] {} (u45)

 	(u49) edge node [above] {} (u50)
 	(u50) edge [line width= 3pt] (u51)
 	(u51) edge node [above] {} (u52)
 	(u52) edge node [above] {} (u53)

 	(u57) edge node [above] {} (u58)
 	(u58) edge [line width= 3pt] (u59)
 	(u59) edge node [above] {} (u60)
 	(u60) edge node [above] {} (u61)

 	(u1) edge node [above] {} (u9)
 	(u9) edge node [above] {} (u17)
 	(u17) edge node [above] {} (u25)
 	(u25) edge node [above] {} (u33)
 	(u33) edge node [above] {} (u41)
 	(u41) edge (u49)
 	(u49) edge (u57)
 	(u2) edge node [above] {} (u10)
 	(u10) edge node [above] {} (u18)
 	(u18) edge node [above] {} (u26)
 	(u26) edge node [above] {} (u34)
 	(u34) edge node [above] {} (u42)
 	(u42) edge (u50)
 	(u50) edge (u58)
 	(u3) edge node [above] {} (u11)
 	(u11) edge node [above] {} (u19)
 	(u19) edge node [above] {} (u27)
 	(u27) edge [line width= 3pt] (u35)
 	(u35) edge node [above] {} (u43)
 	(u43) edge (u51)
 	(u51) edge (u59)
 	(u4) edge node [above] {} (u12)
 	(u12) edge node [above] {} (u20)
 	(u20) edge node [above] {} (u28)
 	(u28) edge node [above] {} (u36)
 	(u36) edge node [above] {} (u44)
 	(u44) edge (u52)
 	(u52) edge (u60)
 	(u5) edge node [above] {} (u13)
 	(u13) edge node [above] {} (u21)
 	(u21) edge node [above] {} (u29)
 	(u29) edge node [above] {} (u37)
 	(u37) edge node [above] {} (u45)
 	(u45) edge (u53)
 	(u53) edge (u61)	
 	;
 	\end{tikzpicture}
 	
	\caption{An almost horizontal slicing}
	\label{fig:gridh02}
\end{figure}

\begin{figure}
	\centering
	\begin{tikzpicture}
 	\vertex (u1) at (0,4) []{};
 	\vertex (u2) at (0,3) []{};
 	\vertex (u3) at (0,2) []{};
 	\vertex (u4) at (0,1) []{};
 	\vertex (u5) at (0,0) []{};
 	
 	\vertex (u9) at (1,4) []{};
 	\vertex (u10) at (1,3) []{};
 	\vertex (u11) at (1,2) []{};
 	\vertex (u12) at (1,1) []{};
 	\vertex (u13) at (1,0) []{};

 	\vertex (u17) at (2,4) []{};
 	\vertex (u18) at (2,3) []{};
 	\vertex (u19) at (2,2) []{};
 	\vertex (u20) at (2,1) []{};
 	\vertex (u21) at (2,0) []{};

 	\vertex (u25) at (3,4) []{};
 	\vertex (u26) at (3,3) []{};
 	\vertex (u27) at (3,2) []{};
 	\vertex (u28) at (3,1) []{};
 	\vertex (u29) at (3,0) []{};

 	\vertex (u33) at (4,4) []{};
 	\vertex (u34) at (4,3) []{};
 	\vertex (u35) at (4,2) []{};
 	\vertex (u36) at (4,1) []{};
 	\vertex (u37) at (4,0) []{};

 	\vertex (u41) at (5,4) []{};
 	\vertex (u42) at (5,3) []{};
 	\vertex (u43) at (5,2) []{};
 	\vertex (u44) at (5,1) []{};
 	\vertex (u45) at (5,0) []{};

 	\vertex (u49) at (6,4) []{};
 	\vertex (u50) at (6,3) []{};
 	\vertex (u51) at (6,2) []{};
 	\vertex (u52) at (6,1) []{};
 	\vertex (u53) at (6,0) []{};

 	\path
 	(u1) edge node [above] {} (u2)
 	(u2) edge node [above] {} (u3)
 	(u3) edge node [above] {} (u4)
 	(u4) edge node [above] {} (u5)

 	(u9) edge node [above] {} (u10)
 	(u10) edge node [above] {} (u11)
 	(u11) edge node [above] {} (u12)
 	(u12) edge node [above] {} (u13)

 	(u17) edge node [above] {} (u18)
 	(u18) edge node [above] {} (u19)
 	(u19) edge node [above] {} (u20)
 	(u20) edge node [above] {} (u21)

 	(u25) edge node [above] {} (u26)
 	(u26) edge [line width= 3pt] (u27)
 	(u27) edge [line width= 3pt] (u28)
 	(u28) edge node [above] {} (u29)

 	(u33) edge node [above] {} (u34)
 	(u34) edge node [above] {} (u35)
 	(u35) edge node [above] {} (u36)
 	(u36) edge node [above] {} (u37)

 	(u41) edge node [above] {} (u42)
 	(u42) edge node [above] {} (u43)
 	(u43) edge node [above] {} (u44)
 	(u44) edge node [above] {} (u45)

 	(u49) edge node [above] {} (u50)
 	(u50) edge node [above] {} (u51)
 	(u51) edge node [above] {} (u52)
 	(u52) edge node [above] {} (u53)

 	(u1) edge node [above] {} (u9)
 	(u9) edge node [above] {} (u17)
 	(u17) edge node [above] {} (u25)
 	(u25) edge [line width= 3pt] (u33)
 	(u33) edge node [above] {} (u41)
 	(u41) edge (u49)
 	
 	(u2) edge node [above] {} (u10)
 	(u10) edge node [above] {} (u18)
 	(u18) edge node [above] {} (u26)
 	(u26) edge [line width= 3pt] (u34)
 	(u34) edge node [above] {} (u42)
 	(u42) edge (u50)
 	
 	(u3) edge node [above] {} (u11)
 	(u11) edge node [above] {} (u19)
 	(u19) edge [line width= 3pt] (u27)
 	(u27) edge [line width= 3pt] (u35)
 	(u35) edge node [above] {} (u43)
 	(u43) edge (u51)
 	
 	(u4) edge node [above] {} (u12)
 	(u12) edge node [above] {} (u20)
 	(u20) edge [line width= 3pt]  (u28)
 	(u28) edge node [above] {} (u36)
 	(u36) edge node [above] {} (u44)
 	(u44) edge (u52)
 	
 	(u5) edge node [above] {} (u13)
 	(u13) edge node [above] {} (u21)
 	(u21) edge [line width= 3pt] (u29)
 	(u29) edge node [above] {} (u37)
 	(u37) edge node [above] {} (u45)
 	(u45) edge (u53)
 	;
 	\end{tikzpicture}
 	
	\caption{An almost vertical slicing}
	\label{fig:gridv03}
\end{figure}

\begin{figure}
	\centering
	\begin{tikzpicture}
 	\vertex (u1) at (0,4) []{};
 	\vertex (u2) at (0,3) []{};
 	\vertex (u3) at (0,2) []{};
 	\vertex (u4) at (0,1) []{};
 	\vertex (u5) at (0,0) []{};
 	
 	\vertex (u9) at (1,4) []{};
 	\vertex (u10) at (1,3) []{};
 	\vertex (u11) at (1,2) []{};
 	\vertex (u12) at (1,1) []{};
 	\vertex (u13) at (1,0) []{};

 	\vertex (u17) at (2,4) []{};
 	\vertex (u18) at (2,3) []{};
 	\vertex (u19) at (2,2) []{};
 	\vertex (u20) at (2,1) []{};
 	\vertex (u21) at (2,0) []{};

 	\vertex (u25) at (3,4) []{};
 	\vertex (u26) at (3,3) []{};
 	\vertex (u27) at (3,2) []{};
 	\vertex (u28) at (3,1) []{};
 	\vertex (u29) at (3,0) []{};

 	\vertex (u33) at (4,4) []{};
 	\vertex (u34) at (4,3) []{};
 	\vertex (u35) at (4,2) []{};
 	\vertex (u36) at (4,1) []{};
 	\vertex (u37) at (4,0) []{};

 	\vertex (u41) at (5,4) []{};
 	\vertex (u42) at (5,3) []{};
 	\vertex (u43) at (5,2) []{};
 	\vertex (u44) at (5,1) []{};
 	\vertex (u45) at (5,0) []{};

 	\vertex (u49) at (6,4) []{};
 	\vertex (u50) at (6,3) []{};
 	\vertex (u51) at (6,2) []{};
 	\vertex (u52) at (6,1) []{};
 	\vertex (u53) at (6,0) []{};

 	\path
 	(u1) edge node [above] {} (u2)
 	(u2) edge node [above] {} (u3)
 	(u3) edge [line width= 3pt]  (u4)
 	(u4) edge node [above] {} (u5)

 	(u9) edge node [above] {} (u10)
 	(u10) edge node [above] {} (u11)
 	(u11) edge [line width= 3pt]  (u12)
 	(u12) edge node [above] {} (u13)

 	(u17) edge node [above] {} (u18)
 	(u18) edge node [above] {} (u19)
 	(u19) edge [line width= 3pt] (u20)
 	(u20) edge node [above] {} (u21)

 	(u25) edge node [above] {} (u26)
 	(u26) edge [line width= 3pt] (u27)
 	(u27) edge [line width= 3pt] (u28)
 	(u28) edge node [above] {} (u29)

 	(u33) edge node [above] {} (u34)
 	(u34) edge [line width= 3pt] (u35)
 	(u35) edge node [above] {} (u36)
 	(u36) edge node [above] {} (u37)

 	(u41) edge node [above] {} (u42)
 	(u42) edge [line width= 3pt] (u43)
 	(u43) edge node [above] {} (u44)
 	(u44) edge node [above] {} (u45)

 	(u49) edge node [above] {} (u50)
 	(u50) edge [line width= 3pt] (u51)
 	(u51) edge node [above] {} (u52)
 	(u52) edge node [above] {} (u53)

 	(u1) edge node [above] {} (u9)
 	(u9) edge node [above] {} (u17)
 	(u17) edge node [above] {} (u25)
 	(u25) edge node [above] {} (u33)
 	(u33) edge node [above] {} (u41)
 	(u41) edge (u49)
 	
 	(u2) edge node [above] {} (u10)
 	(u10) edge node [above] {} (u18)
 	(u18) edge node [above] {} (u26)
 	(u26) edge node [above] {} (u34)
 	(u34) edge node [above] {} (u42)
 	(u42) edge (u50)
 	
 	(u3) edge node [above] {} (u11)
 	(u11) edge node [above] {} (u19)
 	(u19) edge [line width= 3pt] (u27)
 	(u27) edge [line width= 3pt] (u35)
 	(u35) edge node [above] {} (u43)
 	(u43) edge (u51)
 	
 	(u4) edge node [above] {} (u12)
 	(u12) edge node [above] {} (u20)
 	(u20) edge node [above] {}  (u28)
 	(u28) edge node [above] {} (u36)
 	(u36) edge node [above] {} (u44)
 	(u44) edge (u52)
 	
 	(u5) edge node [above] {} (u13)
 	(u13) edge node [above] {} (u21)
 	(u21) edge node [above] {} (u29)
 	(u29) edge node [above] {} (u37)
 	(u37) edge node [above] {} (u45)
 	(u45) edge (u53)
 	;
 	\end{tikzpicture}
 	
	\caption{An almost horizontal slicing}
	\label{fig:gridh03}
\end{figure}
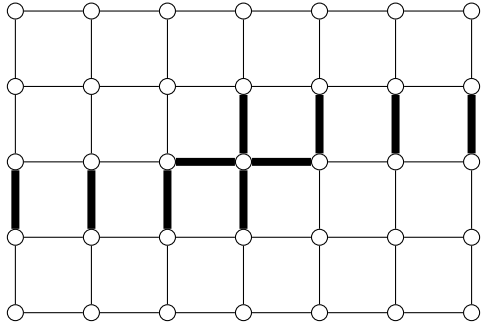

	\subsection*{Summary of the Analysis}
 	Thus from the five scenarios we analyzed, we come to realize the following.
 	
 	\begin{enumerate}
 		\item[\textbf{R1}.]  When $n$ is even, the vertical slicing of the grid graph would provide us with two non-adjacent subgraphs, $A_p$ and $B_p$, of equal size leaving $m$ edges.
 		\item[\textbf{R2}.] When $m$ is even, a similar horizontal slicing of the grid graph would provide us with two non-adjacent subgraphs, $A_p$ and $B_p$, of equal size leaving $n$ edges.
 		\item[\textbf{R3}.] When $n$ is odd $m$ is even, an almost vertical slicing of the grid graph would provide us with two non-adjacent subgraphs, $A_p$ and $B_p$, of equal size leaving $m+1$ edges.
 		\item[\textbf{R4}.] When $m$ is odd and $n$ is even a similar horizontal slicing of the grid graph would provide us with two non-adjacent subgraphs, $A_p$ and $B_p$, of equal size leaving $n+1$ edges.
 		\item[\textbf{R5}.] When $n$ and $m$ are odd, an almost vertical slicing of the grid graph would provide us with two non-adjacent subgraphs, $A_p$ and $B_p$, of equal size leaving $m+3$ edges.
 		\item[\textbf{R6}.] When $n$ and $m$ are odd, an almost horizontal slicing of the grid graph would provide us with two non-adjacent subgraphs, $A_p$ and $B_p$, of equal size leaving $n+3$ edges.
 	\end{enumerate}

 	We now consider non-trivial grids $P_n \times P_m$ with $n > 1$ and $m > 1$.

 	The grid graph $P_n \times P_m$ has $nm$ vertices and $m(n-1)+(m-1)n$=$2mn-(m+n)$ edges. 
 	
 	We find the line completion number by considering the various values of $n$ and $m$.
 	
 	\section*{Case 3: when $n$ and $m$ are even}
 	
 	\subsection*{Subcase 3.1: $m \le n$}
	
 From the Observation \textbf{R1}, we see that a vertical slicing would give use two optimal subgraphs $A_p$ and $B_p$ of size $\frac{2mn-(m+n)-m}{2}$ that are not adjacent to each other and one $mK_2$. If we add any edge from $mK_2$ to any of the subgraphs $A_p$ or $B_p$, they become adjacent. Hence, the size of $A_p$= the size $B_p$ =$mn-m-\frac{n}{2}$.
 
 If we take the horizontal slicing as given in Observation \textbf{R2} leave us with two subgraphs $A_p$ and $B_p$, each of size $mn-n-\frac{m}{2}$.
 	
 Now, 
 	\begin{align*}
 	mn-n-\frac{m}{2}-\big(mn-m-\frac{n}{2}\big)&= \frac{m-n}{2}\\
 	&\begin{cases}
 	=0 &\text{ if }n=m\\
 	<0 &\text{ if }n>m
 	\end{cases}
 	\end{align*}
 	
 	Hence, the vertical slicing gives us more edges. i. e., $mn-m-\frac{n}{2}$ edges. 
 	
 	Thus the line completion number of $P_n \times P_m$ in this case is $mn-m-\frac{n}{2}+1$= $mn+1-\frac{m+n}{2}-\frac{m}{2}$.
 	
 	\subsection*{Subcase 3.2: $m \ge n$}
  		
 	We use the same line of argument as in the previous case and see that the line completion number is obtained from a horizontal slicing. Thus the line completion number of $P_n \times P_m$ is $mn+1-\frac{m+n}{2}-\frac{n}{2}$.
 	
 	Since $P_m \times P_n$ is isomorphic to $P_n \times P_m$, we can have a general formula for the line completion number of $P_n \times P_m$. i. e., 
 	
 	\[
 	lc(P_n \times P_m)= mn+1-\frac{m+n}{2}-\frac{\min{(m, n)}}{2}.
 	\]

 	\section*{Case 4: when $n \ne 1$ and $m \ne 1$ are odd}

\subsection*{Subcase 4.1: $m < n$}

	From the Observation \textbf{R5}, an almost vertical slicing would give use two optimal subgraphs $A_p$ and $B_p$ of equal size leaving exactly $m+3$ edges. Hence, the size of subgraph $A_p$ is 
	\[
	\frac{2mn-(m+n)-(m+3)}{2} = mn-m-\frac{(n+3)}{2}.
	\]
	
	Similarly, an almost horizontal slicing as observed in \textbf{R6}, would give us two optimal subgraphs $A_p$ and $B_p$ of equal size leaving exactly $n+3$ edges.
	Hence, the size of subgraph $A_p$ is 
	\[
	\frac{2mn-(m+n)-(n+3)}{2} = mn-n-\frac{(m+3)}{2}.
	\]
	Now,
	\begin{align*}
	mn-m-\frac{(n+3)}{2}-\big(mn-n-\frac{(m+3)}{2}\big)&= \frac{n-m}{2}\\
	& >0 \text{ if }n>m.
	\end{align*}
	
	Hence, the vertical slicing gives us more edges. Therefore, the line completion number of the $P_n \times P_m$ in this case is
	\[
	mn-m-\frac{(n+3)}{2}+1=mn-\frac{(m+n)}{2}-\frac{(m+1)}{2}.
	\]
	\subsection*{Subcase 4.2: $m > n$}
	
	Since $P_m \times P_n$ is isomorphic to $P_n \times P_m$, we can conclude that,
		
	\[
	lc(P_n \times P_m)=mn-\frac{(m+n)}{2}-\frac{(n+1)}{2}.
	\]
	
	A combined formula for this case is thus obtained as, 
	
	\[
	lc(P_n \times P_m)= mn-\frac{(m+n)}{2}-\frac{\min{(m, n)}+1}{2}.
	\]
	 	\section*{Case 5: when $n \ne 1$ and $m \ne 1$ are of opposite parity}
	
	\subsection*{Subcase 5.1: $m < n$}
	
	\subsubsection*{Subsubcase 5.1.1: $m$ is odd and $n$ is even}		
	
	Using Observation \textbf{R1}, we do a vertical slicing of the grid graph and obtain two optimal subgraphs each of size
	
	\[
	mn-m-\frac{n}{2}.
	\]
	
	Using Observation \textbf{R4}, we do an almost horizontal slicing of the grid graph and obtain two subgraphs each of size 
	
	\[
	mn-n-\frac{(m+1)}{2}.
	\]
	
	However, it is given that $m < n$. Therefore we have,
	
	\[
	mn-m-\frac{n}{2}- \big(mn-n-\frac{(m+1)}{2}\big)= \frac{(n-m)}{2}+\frac{1}{2}>0.
	\]
	
	Thus the vertical slicing provides the larger size. Hence, the line completion number is, 
	
	\[
	\big(mn-m-\frac{n}{2}\big)+1=mn+1-m-\frac{n}{2}.
	\]
	\subsubsection*{Subsubcase 5.1.2: $m$ is even and $n$ is odd}	
	
	Using Observation \textbf{R3}, we do an almost vertical slicing of the grid graph and obtain two optimal subgraphs each of size
	
	\[
	mn-m-\frac{(n+1)}{2}.
	\]
	
	Using Observation \textbf{R2}, we do a horizontal slicing of the grid graph and obtain two subgraphs each of size 
	
	\[
	mn-n-\frac{m}{2}.
	\]
	
	However, it is given that $m < n$. Therefore we have,
	
	\[
	mn-m-\frac{(n+1)}{2}- \big(mn-n-\frac{m}{2}\big)= \frac{(n-m)}{2}-\frac{1}{2}\ge0 \text{ as (n-m) is at least 1}.
	\]
	
	Thus the vertical slicing provides the larger size. Hence, the line completion number is, 
	
	\[
	\big(mn-m-\frac{(n+1)}{2}\big)+1=mn+1-m-\frac{(n+1)}{2}.
	\]

	A combined formula for Subcase 5.1 is thus
	\[
	lc(P_n \times P_m)=mn+1-m-\lceil\frac{n}{2}\rceil.
	\]
	\subsection*{Subcase 5.2: $m > n$}
	
	\subsubsection*{Subsubcase 5.2.1: $m$ is odd and $n$ is even}		
	
	Using Observation \textbf{R1}, we do a vertical slicing of the grid graph and obtain two optimal subgraphs each of size
	
	\[
	mn-m-\frac{n}{2}.
	\]
	
	Using Observation \textbf{R4}, we do an almost horizontal slicing of the grid graph and obtain two subgraphs each of size 
	\[
	mn-n-\frac{(m+1)}{2}.
	\]
		However, it is given that $m > n$. Therefore we have,
	\[
	mn-m-\frac{n}{2}- \big(mn-n-\frac{(m+1)}{2}\big)= \frac{(n-m)}{2}+\frac{1}{2}\le0 \text{ as (n-m) is at most -1}.
	\]
	Thus the vertical slicing provides the larger size. Hence, the line completion number is, 
	\[
	\big(mn-n-\frac{(m+1)}{2}\big)+1=mn+1-n-\frac{(m+1)}{2}.
	\]
		\subsubsection*{Subsubcase 5.2.2: $m$ is even and $n$ is odd}

	Using Observation \textbf{R3}, we do an almost vertical slicing of the grid graph and obtain two optimal subgraphs, as in Case 3, each of size
	
	\[
	mn-m-\frac{(n+1)}{2}.
	\]
	
	Using Observation \textbf{R2}, we do a horizontal slicing of the grid graph and obtain two subgraphs, as in Case 2, each of size 
	
	\[
	mn-n-\frac{(m)}{2}.
	\]
	
	However, it is given that $m > n$. Therefore we have,
	
	\[
	mn-m-\frac{(n+1)}{2}- \big(mn-n-\frac{m}{2}\big)= \frac{(n-m)}{2}-\frac{1}{2}<0 \text{ as (n-m) is at most -1}.
	\]
	
	Thus the horizontal slicing provides the larger size. Hence, the line completion number is, 
	
	\[
	\big(mn-n-\frac{m}{2}\big)+1=mn+1-n-\frac{m}{2}.
	\]
	
	A combined formula for Subcase 5.2 is thus
	\[
	lc(P_n \times P_m)=mn+1-n-\lceil\frac{m}{2}\rceil.
	\]
	
	Subcases 5.1 and 5.2 can be combined to get the formula for the line completion number of the grid graph when $n$ and $m$ of opposite parity as,
	
	\[
	lc(P_n \times P_m)=mn+1-\min{(m, n)}-\lceil\frac{\max{(m, n)}}{2}\rceil.
	\]
	
\end{proof}

\section{Illustrations}
We see some illustrations now.
\begin{example}\label{exeveeven}
	$m$ and $n$ are even
\end{example}
Consider the graph $P_4 \times P_6$.
\begin{figure}[h!]
	\centering
	\begin{tikzpicture}
	\vertex (u1) at (0,0) []{};
	\vertex (u2) at (1,0) []{};
	\vertex (u3) at (2,0) []{};
	\vertex (u4) at (3,0) []{};
	\vertex (u5) at (4,0) []{};
	\vertex (u6) at (5,0) []{};
	\vertex (u7) at (0,1) []{};
	\vertex (u8) at (1,1) []{};
	\vertex (u9) at (2,1) []{};
	\vertex (u10) at (3,1) []{};
	\vertex (u11) at (4,1) []{};
	\vertex (u12) at (5,1) []{};
	\vertex (u13) at (0,2) []{};
	\vertex (u14) at (1,2) []{};
	\vertex (u15) at (2,2) []{};
	\vertex (u16) at (3,2) []{};
	\vertex (u17) at (4,2) []{};
	\vertex (u18) at (5,2) []{};
	\vertex (u19) at (0,3) []{};
	\vertex (u20) at (1,3) []{};
	\vertex (u21) at (2,3) []{};
	\vertex (u22) at (3,3) []{};
	\vertex (u23) at (4,3) []{};
	\vertex (u24) at (5,3) []{};
	\path
	(u1) edge[line width= 2pt] (u2)
	(u2) edge[line width= 2pt] (u3)
	(u3) edge node [above] {$e_{43}$} (u4)
	(u4) edge[line width= 2pt] (u5)
	(u5) edge[line width= 2pt] (u6)
	(u7) edge[line width= 2pt] (u8)
	(u8) edge[line width= 2pt] (u9)
	(u9) edge node [above] {$e_{33}$} (u10)
	(u10) edge[line width= 2pt] (u11)
	(u11) edge[line width= 2pt] (u12)
	(u13) edge[line width= 2pt] (u14)
	(u14) edge[line width= 2pt] (u15)
	(u15) edge node [above] {$e_{23}$}(u16)
	(u16) edge[line width= 2pt] (u17)
	(u17) edge[line width= 2pt] (u18)
	(u19) edge[line width= 2pt] (u20)
	(u20) edge[line width= 2pt] (u21)
	(u21) edge node [above] {$e_{13}$}(u22)
	(u22) edge[line width= 2pt] (u23)
	(u23) edge[line width= 2pt] (u24)
	(u1) edge[line width= 2pt] (u7)
	(u7) edge[line width= 2pt] (u13)
	(u13) edge[line width= 2pt] (u19)
	(u2) edge[line width= 2pt] (u8)
	(u8) edge[line width= 2pt] (u14)
	(u14) edge[line width= 2pt] (u20)
	(u3) edge[line width= 2pt] (u9)
	(u9) edge[line width= 2pt] (u15)
	(u15) edge[line width= 2pt] (u21)
	(u4) edge[line width= 2pt] (u10)
	(u10) edge[line width= 2pt] (u16)
	(u16) edge[line width= 2pt] (u22)
	(u5) edge[line width= 2pt] (u11)
	(u11) edge[line width= 2pt] (u17)
	(u17) edge[line width= 2pt] (u23)
	(u6) edge[line width= 2pt] (u12)
	(u12) edge[line width= 2pt] (u18)
	(u18) edge[line width= 2pt] (u24)
	;
	\end{tikzpicture}
	\caption{\label{fig:4}$P_4 \times P_6$}
\end{figure}
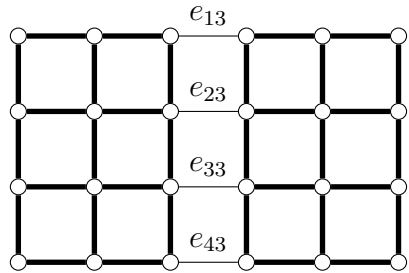
$P_4 \times P_6$ will have 24 vertices. We partition the vertex set into two equal subsets. Let A and B be the sets. We put vertices of the first 3 $P_4$ paths in set $A$. Now the subgraph induced by set $A$ will be $P_4 \times P_3$. Now we take the  vertices of 4th, 5th and 6th paths in set $B$. Set $B$ also induce a subgraph $P_4 \times P_3$. Let $C$ and $D$ be the edge set of $P_4 \times P_3$. From Figure \ref{fig:4}, we can see that $C$ and $D$ are not adjacent in $L_{17}(P_4 \times P_6)$. Now, if we add any of the edges $e_{13},e_{23},e_{33},e_{43}$ to set $C$ and $D$
they become adjacent. Therefore the $lc(P_4 \times P_6)$ is 18.

\begin{example}
	$P_4 \times P_5$
\end{example}

Let us consider $P_4 \times P_5$. We partition the vertex set as in Figure  \ref{fig:6}.

\begin{figure}[h!]
	\centering
	\begin{tikzpicture}
	
	\vertex (u1) at (0,0) []{};
	\vertex (u2) at (1,0) []{};
	\vertex (u3) at (2,0) []{};
	\vertex (u4) at (3,0) []{};
	\vertex (u5) at (4,0) []{};
	
	\vertex (u7) at (0,1) []{};
	\vertex (u8) at (1,1) []{};
	\vertex (u9) at (2,1) []{};
	\vertex (u10) at (3,1) []{};
	\vertex (u11) at (4,1) []{};
	
	\vertex (u13) at (0,2) []{};
	\vertex (u14) at (1,2) []{};
	\vertex (u15) at (2,2) []{};
	\vertex (u16) at (3,2) []{};
	\vertex (u17) at (4,2) []{};
	
	\vertex (u19) at (0,3) []{};
	\vertex (u20) at (1,3) []{};
	\vertex (u21) at (2,3) []{};
	\vertex (u22) at (3,3) []{};
	\vertex (u23) at (4,3) []{};
	
	\path
	(u1) edge[line width= 2pt] (u2)
	(u2) edge node [above] {$e_{42}$} (u3)
	(u3) edge[line width= 2pt] (u4)
	(u4) edge[line width= 2pt] (u5)
	
	(u7) edge[line width= 2pt] (u8)
	(u8) edge node [above] {$e_{32}$} (u9)
	(u9) edge[line width= 2pt] (u10)
	(u10) edge[line width= 2pt] (u11)
	
	(u13) edge[line width= 2pt] (u14)
	(u14) edge[line width= 2pt] (u15)
	(u15) edge node [above] {$e_{23}$} (u16)
	(u16) edge[line width= 2pt] (u17)
	
	(u19) edge[line width= 2pt] (u20)
	(u20) edge[line width= 2pt] (u21)
	(u21) edge node [above] {$e_{13}$}(u22)
	(u22) edge[line width= 2pt] (u23)
	
	(u1) edge[line width= 2pt] (u7)
	(u7) edge[line width= 2pt] (u13)
	(u13) edge[line width= 2pt] (u19)
	
	(u2) edge[line width= 2pt] (u8)
	(u8) edge[line width= 2pt] (u14)
	(u14) edge[line width= 2pt] (u20)
	
	(u3) edge[line width= 2pt] (u9)
	(u9) edge node [right] {$f_{23}$} (u15)
	(u15) edge[line width= 2pt] (u21)
	
	(u4) edge[line width= 2pt] (u10)
	(u10) edge[line width= 2pt] (u16)
	(u16) edge[line width= 2pt] (u22)
	
	(u5) edge[line width= 2pt] (u11)
	(u11) edge[line width= 2pt]  (u17)
	(u17) edge[line width= 2pt] (u23)
	
	;
	\end{tikzpicture}
	\caption{\label{fig:6}$P_4 \times P_5$}
\end{figure}
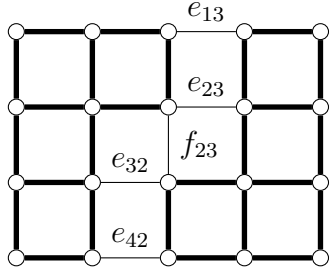
Since both of them are same we prefer the partition in figure 4.3. Now the subgraph induced by these two partitions is $P_2 \times P_5$. Edge set of both copies of $P_2 \times P_5$ will contain 13 edges. Therefore line completion number of $P_4 \times P_5$ = 13+1 = 14.

\begin{example}
	$P_5 \times P_7$
\end{example}
Consider the example $P_5 \times P_7$.
\begin{figure}[h!]
	\centering
	\begin{tikzpicture}
	
	\vertex (u1) at (0,0) []{};
	\vertex (u2) at (1,0) []{};
	\vertex (u3) at (2,0) []{};
	\vertex (u4) at (3,0) []{};
	\vertex (u5) at (4,0) []{};
	\vertex (u6) at (5,0) []{};
	\vertex (u7) at (6,0) []{};
	
	\vertex (u8) at (0,1) []{};
	\vertex (u9) at (1,1) []{};
	\vertex (u10) at (2,1) []{};
	\vertex (u11) at (3,1) []{};
	\vertex (u12) at (4,1) []{};
	\vertex (u13) at (5,1) []{};
	\vertex (u14) at (6,1) []{};
	
	\vertex (u15) at (0,2) []{};
	\vertex (u16) at (1,2) []{};
	\vertex (u17) at (2,2) []{};
	\vertex (u18) at (3,2) []{};
	\vertex (u19) at (4,2) []{};
	\vertex (u20) at (5,2) []{};
	\vertex (u21) at (6,2) []{};
	
	\vertex (u22) at (0,3) []{};
	\vertex (u23) at (1,3) []{};
	\vertex (u24) at (2,3) []{};
	\vertex (u25) at (3,3) []{};
	\vertex (u26) at (4,3) []{};
	\vertex (u27) at (5,3) []{};
	\vertex (u28) at (6,3) []{};
	
	\vertex (u29) at (0,4) []{};
	\vertex (u30) at (1,4) []{};
	\vertex (u31) at (2,4) []{};
	\vertex (u32) at (3,4) []{};
	\vertex (u33) at (4,4) []{};
	\vertex (u34) at (5,4) []{};
	\vertex (u35) at (6,4) []{};
	
	\path
	(u1) edge[line width= 2pt] (u2)
	(u2) edge[line width= 2pt] (u3)
	(u3) edge node [below] {$e_{53}$} (u4)
	(u4) edge[line width= 2pt] (u5)
	(u5) edge[line width= 2pt] (u6)
	(u6) edge[line width= 2pt] (u7)
	
	(u8) edge[line width= 2pt] (u9)
	(u9) edge[line width= 2pt] (u10)
	(u10) edge node [below] {$e_{43}$} (u11)
	(u11) edge[line width= 2pt] (u12)
	(u12) edge[line width= 2pt] (u13)
	(u13) edge[line width= 2pt] (u14)
	
	(u15) edge[line width= 2pt] (u16)
	(u16) edge[line width= 2pt] (u17)
	(u17) edge node [below] {$e_{33}$} (u18)
	(u18) edge node [above] {$e_{34}$} (u19)
	(u19) edge[line width= 2pt] (u20)
	(u20) edge[line width= 2pt] (u21)
	
	(u22) edge[line width= 2pt] (u23)
	(u23) edge[line width= 2pt] (u24)
	(u24) edge[line width= 2pt] (u25)
	(u25) edge node [above] {$e_{24}$}(u26)
	(u26) edge[line width= 2pt] (u27)
	(u27) edge[line width= 2pt] (u28)
	
	(u29) edge[line width= 2pt] (u30)
	(u30) edge[line width= 2pt] (u31)
	(u31) edge[line width= 2pt] (u32)
	(u32) edge node [above] {$e_{14}$} (u33)
	(u33) edge[line width= 2pt] (u34)
	(u34) edge[line width= 2pt] (u35)
	
	(u1) edge[line width= 2pt] (u8)
	(u8) edge[line width= 2pt] (u15)
	(u15) edge[line width= 2pt] (u22)
	(u22) edge[line width= 2pt] (u29)
	
	(u2) edge[line width= 2pt] (u9)
	(u9) edge[line width= 2pt] (u16)
	(u16) edge[line width= 2pt] (u23)
	(u23) edge[line width= 2pt] (u30)
	
	(u3) edge[line width= 2pt] (u10)
	(u10) edge[line width= 2pt] (u17)
	(u17) edge[line width= 2pt] (u24)
	(u24) edge[line width= 2pt] (u31)
	
	(u4) edge[line width= 2pt] (u11)
	(u11) edge node [right] {$f_{34}$}(u18)
	(u18) edge node [left] {$f_{24}$}(u25)
	(u25) edge[line width= 2pt] (u32)
	
	(u5) edge[line width= 2pt] (u12)
	(u12) edge[line width= 2pt] (u19)
	(u19) edge[line width= 2pt] (u26)
	(u26) edge[line width= 2pt] (u33)
	
	(u6) edge[line width= 2pt] (u13)
	(u13) edge[line width= 2pt] (u20)
	(u20) edge[line width= 2pt] (u27)
	(u27) edge[line width= 2pt] (u34)
	
	(u7) edge[line width= 2pt] (u14)
	(u14) edge[line width= 2pt] (u21)
	(u21) edge[line width= 2pt] (u28)
	(u28) edge[line width= 2pt] (u35)
	;
	\end{tikzpicture}
	\caption{\label{fig:7}{$P_5 \times P_7$}}
\end{figure}
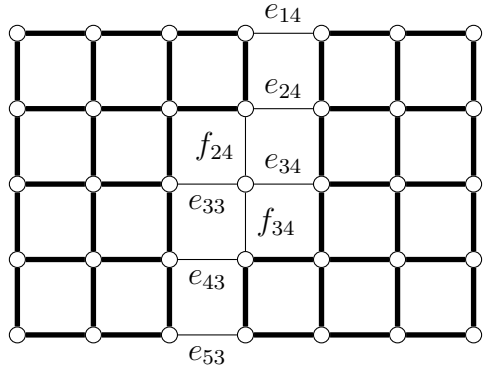
Number of vertices in $P_5 \times P_7$ is 35. We cannot partition the vertex set equally. Let A and B be the partitions. Let $|A| = 18$ and $|B| = 17$. Now if we remove the 3rd vertex of 4th $P_5$ path from set $A$, we get two equal partitions. Therefore, the number of edges in the subgraphs induced by both $A$ and $B = 25$. From Figure \ref{fig:7}, it is clear that $L_{25}(P_5 \times P_7)$ is not complete.  Now if we add one more edge to those corresponding edge sets they become adjacent. Therefore the line completion number is 25+1 = 26.
\section*{Acknowledgement}

We express our sincere gratitude to the reviews and comments given by the scholars when we uploaded the manuscript on arxiv.org. Our special thanks to Michel Marcus. We thank the anonymous referees for the comments and corrections provided. 
      
    \end{document}